\documentclass{amsart}
\usepackage{amsthm,amssymb,graphicx,hyperref}
\newtheorem{theorem}{Theorem}[section]

\newtheorem{corollary}[theorem]{Corollary}

\newtheorem{lemma}[theorem]{Lemma}
\newtheorem*{lemma*}{Lemma}

\newtheorem{proposition}[theorem]{Proposition}
\newtheorem{question}[theorem]{Question}

\theoremstyle{remark}

\newcommand{\bfR}{\mathbf{R}}

\newcommand{\cgF}{\mathcal{F}}

\newcommand{\cgL}{\mathcal{L}}

\newcommand{\cgI}{\mathcal{I}}
\newcommand{\cgJ}{\mathcal{J}}

\newcommand{\bbB}{\mathbb{B}}
\newcommand{\bbE}{\mathbb{E}}

\newcommand{\bbP}{\mathbb{P}}
\newcommand{\bbR}{\mathbb{R}}

\newcommand{\sdim}{\operatorname{sdim}}

\newcommand{\GLR}{\operatorname{GLR}}
\newcommand{\Min}{\operatorname{Min}}
\newcommand{\Max}{\operatorname{Max}}

\newcommand{\sa}{\operatorname{sa}}
\newcommand{\se}{\operatorname{se}}

\newcommand{\wttbin}{\operatorname{bin}}
\newcommand{\wttbcn}{\operatorname{bcn}}

\begin{document}

\title[RANDOM POSETS]{Random Bipartite Posets and Extremal Problems}

\author[BIR\'{O}]{Csaba Bir\'{o}}
\address{Department of Mathematics\\
 University of Louisville\\
 Louisville, Kentucky 40292}
\email{csaba.biro@louisville.edu}

\author[HAMBURGER]{Peter Hamburger}
\address{Purdue University, IN\\}
\email{hamburge@pfw.edu}

\author[KIERSTEAD]{H. A. Kierstead}
\address{School of Mathematical and Statistical Sciences\\
  Arizona State University\\
  Tempe, Arizona 85287}
\email{kierstead@asu.edu}

\author[P\'{O}R]{Attila P\'{o}r}
\address{Department of Mathematics\\
 Western Kentucky University\\
 Bowling Green, Kentucky 42101}
\email{attila.por@wku.edu}

\author[TROTTER]{William T. Trotter}
\address{School of Mathematics\\
  Georgia Institute of Technology\\
  Atlanta, Georgia 30332}
\email{trotter@math.gatech.edu}

\author[WANG]{Ruidong Wang}
\address{School of Mathematics\\
  Georgia Institute of Technology\\
  Atlanta, Georgia 30332}
\email{rwang49@math.gatech.edu}

\date{February 21, 2020}



\begin{abstract}
Previously, Erd\H{o}s, Kierstead and Trotter~\cite{ErKiTr91} investigated 
the dimension of random height~$2$ partially ordered sets.  Their 
research was motivated primarily by two goals:  (1)~analyzing the relative 
tightness of the F\"{u}redi-Kahn upper bounds on dimension in terms of 
maximum degree; and (2)~developing machinery for estimating the
expected dimension of a random labeled poset on $n$ points.  For these
reasons, most of their effort was focused on the case $0<p\le 1/2$.
While bounds were given for the range $1/2\le p <1$, the relative accuracy of the
results in the original paper deteriorated as $p$ approaches~$1$. 

Motivated by two extremal problems involving conditions that force a 
poset to contain a large standard example, we were compelled to 
revisit this subject, but now with primary emphasis on the range 
$1/2\le p<1$.  Our sharpened analysis shows that as $p$ 
approaches~$1$, the expected value of dimension increases and
then decreases, answering in the negative a question posed in the original paper.
Along the way, we apply inequalities of Talagrand and Janson,
establish connections with latin rectangles and the Euler product function,
and make progress on both extremal problems.
\end{abstract}

\maketitle

\section{Introduction}\label{sec:introduction}

This paper is concerned primarily with the combinatorics of finite \textit{partially ordered
sets}, also called \textit{posets}, but to motivate our line of research,
we start with a brief discussion of analogous
questions for graphs.  For a  graph $G$, let $\omega(G)$ denote the
\textit{clique number} of $G$, the maximum number of vertices in a complete subgraph of
$G$.  Also, let $\chi(G)$ denote the \textit{chromatic number} of $G$, the 
least number of colors required for a proper coloring of $G$. 
Let $N$ denote the set of positive integers, and when
$n\in N$, we write $[n]$ for $\{1,...,n\}$.  

Let $n\in N$, and let $G$ be a graph on $n$ vertices.  Then we have the trivial inequality 
$\omega(G)\le\chi(G)\le n$.  Furthermore, if $\chi(G) = n$, then $\omega(G)=n$.
We then ask whether these statements are ``stable,'' i.e., if $G$ is a graph on $n$
vertices and $\chi(G)$ is close to $n$, must $\omega(G)$ also be close to $n$?  More
formally:

\begin{question}[Question 1 for Graphs]
Does there exist 
a function $f:N\rightarrow N$ such that for every $c\in N$, if $n>f(c)$,
$G$ is a graph on $n$ vertices, and $\chi(G)\ge n-c$, then 
$\omega(G)\ge n-f(c)$?  
\end{question}

The answer to Question~1 for graphs is easily seen to
be yes, since we may take $f(c)=2c$.  To see that this function
satisfies the desired property,  we simply carry out the following
iterative process:\quad While $G$ is not a complete graph, choose two non-adjacent vertices
and remove them.  Each such operation decreases the size of the
graph by~$2$ but lowers the chromatic number by at most~$1$.   So the operation must
halt in at most $c$ steps.

Here is a second question of a related nature.

\begin{question}[Question 2 for Graphs]
For integers $k$ and $n$ with $k\ge2$ fixed
and $n\rightarrow\infty$, what is the maximum value $g(k,n)$ of 
$\chi(G)$ among all graphs $G$ on $n$ vertices with $\omega(G)<k$? 
\end{question}

For example, when $k=3$, a graph $G$ with $\omega(G)<3$ is said to be 
\textit{triangle-free}, and it is well known that $g(3,n)$, the maximum 
chromatic number of a triangle-free graph on $n$ vertices, is $\Theta(\sqrt{n/\log n})$.

The primary goal of this paper is to investigate analogous questions for posets.
We assume that readers are familiar with basic notation and terminology for
posets, including comparable and incomparable pairs of points; chains and antichains;
minimal and maximal elements; and linear extensions.  Beyond these basics, 
we will include all essential notation and terminology for the results
presented in this paper.  

To develop the poset analogue of chromatic number,
we have the following definitions.  Let $P$ be a poset.  
A non-empty family $\cgF=\{L_1,\dots,L_d\}$ of linear extensions of $P$ is called
a \textit{realizer of $P$} when $x\le y$ in $P$ if and only if $x\le y$ in $L_j$
for each $j\in[d]$. Dushnik and Miller~\cite{DusMil41}
defined the \textit{dimension} of a poset $P$, 
denoted $\dim(P)$, as the least positive integer $d$ for which 
there is a realizer $\cgF=\{L_1,\dots,L_d\}$ of $P$. 
Analogies between dimension for posets and chromatic number for graphs have
been widely studied, and indeed the book chapter~\cite{Trot19+} is devoted
entirely to this topic.

For the poset analogue of a clique, we have the following construction.
For an integer $d\ge2$, let $S_d$ be the height~$2$ poset
with $\Min(S_d)=\{a_1,\dots,a_d\}$, $\Max(S_d)=\{a'_1,\dots,
a'_d\}$ and $a_i<a'_j$ in $S_d$ if and only if $i\neq j$.
Clearly, $\dim(S_d)=d$, and posets in the
family $\{S_d:d\ge 2\}$ are called \textit{standard examples}.
For a poset $P$, we define the \textit{standard example number} of
$P$, denoted $\se(P)$, as follows.  Set $\se(P)=1$ if $P$ does not contain a subposet
isomorphic to the standard example $S_2$; otherwise $\se(P)$ is the
largest $d\ge2$ for which $P$ contains a subposet isomorphic to the
standard example $S_d$.  We then have the trivial inequality $\dim(P)\ge\se(P)$.
As is well known, for every $d\ge2$, there is a poset $P$ with $\se(P)=1$ and
$\dim(P)=d$.  Nevertheless, it is of interest to study classes of posets where
large dimension requires large standard example number.

When $n\in N$ and $G$ is a graph on $n$ vertices, the inequality $\chi(G)\le n$ 
is trivial, as is the assertion that the inequality is tight only when $G$ is
a complete graph on $n$ vertices.  The analogous results for posets are more substantive.
Hiraguchi~\cite{Hira55} proved that if $n\ge2$ and $P$ is a poset on $2n+1$
points, then $\dim(P)\le n$.   Kimble~\cite{Kimb73} proved\footnote{We 
refer the reader to the discussion in~\cite{Trot19+} about subtleties of this
proof, and we note that it does not hold when $n=2$ or when $n=3$.} that if $n\ge4$, this
inequality is tight only when $P$ contains the standard example $S_n$.
For the poset analogue of Question~1 for graphs, we then have:

\begin{question}[Question 1 for Posets]
Does there exist a function $\sa:N\rightarrow N$
such that for every $c\in N$, if $n>\sa(c)$,
$P$ is a poset on $2n+1$ points, and $\dim(P)\ge n-c$, then $\se(P)\ge n-\sa(c)$?
\end{question}

In stating Question~1 for posets, we use the notation $\sa(c)$ to remind readers that we
are discussing ``stability analysis.''
Unlike the situation with graphs, we know of no elementary argument to show 
that the function $\sa(c)$ is well defined, ignoring any issue of how fast $\sa(c)$
would have to grow in terms of $c$.  However, in~\cite{BHPT16}, 
Bir\'{o}, Hamburger, P\'{o}r and Trotter showed the function $\sa(c)$ is
well defined and satisfies $\sa(c)=O(c^2)$. Also, they gave
a construction using finite projective planes to show that
$\sa(c)=\Omega(c^{4/3})$.  The research for this paper began with the challenge
of finding the correct exponent on $c$ in the function $\sa(c)$, knowing that the
answer is in the interval $[4/3,2]$.  In this paper, we will raise the lower
bound on this interval to $3/2$.

For the poset analogue of Question 2 for graphs, we have:

\begin{question}[Question 2 for Posets]  
For integers $d$ and $n$ with
$d\ge2$ fixed and $n\rightarrow\infty$, what is the maximum
value $f(d,n)$ of $\dim(P)$ among all posets $P$ on $n$ points with $\se(P)<d$?
\end{question}

Question~2 for posets was first posed in~\cite{Trot89},
and then referenced again in~\cite{ErKiTr91}. Here we obtain
a better result, and we remove the requirement that $d$ be large.
For historical reasons, the value of $f(2,n)$ has been studied---albeit
with different notation and terminology---for many years. No doubt
this results from the fact that the class of posets with standard example
number~$1$ is the class of \textit{interval orders}.
Combining results of several authors (see the discussion
in~\cite{BHPT16}), the value of $f(2,n)$ can be determined to
within an additive error of at most~$5$.  However, as a
crude estimate, we have
\[
 f(2,n)=\lg\lg n + (1/2+o(1))\lg\lg\lg n.
\]

For a fixed value of $d\ge3$, Bir\'{o}, Hamburger and
P\'{o}r~\cite{BiHaPo15} proved that $f(d,n)=o(n)$, but this leaves
open the possibility that $f(d,n)$ behaves in the same slow-growing
manner as $f(2,n)$.  However, we will show in Section~\ref{sec:eps} 
that there is a positive constant $\alpha_d$ so that $f(d,n)=\Omega(n^{\alpha_d})$.

\subsection{Links with Random Bipartite Posets}

Working on Question 1 for posets led us to revisit the following model 
for a random bipartite poset, introduced and
studied by Erd\H{o}s, Kierstead and Trotter~\cite{ErKiTr91}.  
Let $n\in N$ and fix disjoint sets $A$ and $A'$, each of size $n$.  Then 
$\Omega(n,p)$ denotes the probability space consisting of posets $P$ such that
(1)~the ground set of $P$ is $A\cup A'$ with $A\subseteq \Min(P)$ and 
$A'\subseteq \Max(P)$; and (2)~for a pair $(a,a')\in A\times A'$,
set $\bbP(a<a'\text{ in } P)=p$, (in general, $p$ is a function of $n$) with 
events corresponding to distinct pairs independent.  

To place this work in historical perspective, we give here
a brief overview of key results, beginning with a 
discussion of upper bounds on dimension.
For a poset $P$ (of arbitrary height), let $\Delta_U(P)$ denote
the maximum size of sets of the form $U_P(x)=\{y\in P: x\le y\text{ in }P\}$
taken over all elements $x\in P$.  Analogously, $\Delta_D(P)$ is
the maximum size of sets of the form $D_P(x)=\{z\in P:z\le x\text{ in }P\}$.
Then set $\Delta(P)=\max\{\Delta_U(P),\Delta_D(P)\}$.  In~\cite{FurKah86},
F\"{u}redi and Kahn proved that if $\Delta(P)=k$, then
$\dim(P)<50k\log^2 k$.

The first inequality in the following theorem is due to F\"{u}redi and 
Kahn~\cite{FurKah86}.  The second is a quite recent result of 
Scott and Wood~\cite{ScoWoo18+} improving the bound from~\cite{FurKah86}
cited immediately above.  Readers may note that the Lov\'{a}sz local 
lemma~\cite{ErdLov75} was used in both~\cite{FurKah86} and~\cite{ScoWoo18+}.

\begin{theorem}\label{thm:old-ub}
If $u=\Delta_U(P)$, $k=\Delta(P)$ and $|P|=n$, then 
\begin{align*}\label{eqn:FurKah}
\dim(P)&<1+2(u+1)\log n\quad\text{and} \\
\dim(P)&<k\log^{1+o(1)} k
\end{align*}
\end{theorem}

Continuing with upper bounds,
as is well known, almost all labeled posets on $n$ points have the following
structure:\quad  $P$ is the union $\Min(P)\cup A\cup \Max(P)$ of three disjoint
antichains;  the size of $A$ is $(1/2\pm o(1))n$; 
both $\Min(P)$ and $\Max(P)$ have size $(1/4\pm o(1))n$;
and $x<y$ in $P$ whenever $x\in\Min(P)$ and $y\in\Max(P)$.  It is then straightforward 
to show that there is a constant $c_1>0$ such that almost all labeled posets on 
$n$ points have dimension at most $n/4-c_1n/\log n$.

Turning to lower bounds, it is more challenging to find good lower bounds in
either of these two settings.  For example, no explicit construction is known for
a poset $P$ with $\Delta(P)=k$ and $\dim(P) > k+1$ for \emph{any}
value of $k$.  Also, simple counting only shows that almost all labeled
posets on $n$ elements have dimension $\Omega(n/\log n)$.

The following lower bounds are proved in~\cite{ErKiTr91}.
In stating these bounds, we use the standard abbreviation 
$a.a.s.$ for \textit{asymptotically almost surely}.

\begin{theorem}\label{thm:old-lb}
For every $\epsilon>0$, there exists positive constants $\delta_1,\delta_2,\delta_3$ so that
$a.a.s.$,
\begin{equation*}\label{eqn:old-lb}
\dim(P)>\begin{cases}
  \delta_1 pn\log pn &\text{if }n^{-1+\epsilon}<p \le 1/\log n, \text{and}\\
  \max\Bigl\{\delta_2 n, n -\delta_3n/(p\log n)\Bigr\}&\text{if }1/\log n\le p < 
   1-n^{-1+\epsilon}.
\end{cases}
\end{equation*}
\end{theorem}

The first inequality in Theorem~\ref{thm:old-lb} shows that the two upper bounds
in Theorem~\ref{thm:old-ub} are essentially best possible---although 
there remains an $o(\log k)$ multiplicative gap for the second.
The second inequality in Theorem~\ref{thm:old-lb} was used in~\cite{ErKiTr91} to 
show that is a positive constant $c_2>0$ so that almost all labeled posets
on $n$ elements have dimension greater than $n/4 - c_2n/\log n$.  

When $n^{-1}\log^2 n<p\le 1/\log n$, good upper bounds on the expected value of the dimension
of a poset $P\in\Omega(n,p)$ are provided by Theorem~\ref{thm:old-ub}, since
$\Delta(P)$ and $\Delta_U(P)$ are sharply concentrated around $pn$.  For 
the range $1/\log n\le p<1$, the following upper bound is given in~\cite{ErKiTr91}.

\begin{theorem}\label{thm:old-ub-general}
If $\epsilon>0$ and $1/\log n\le p < 1$, then $a.a.s.$,
\[
\dim(P)\le n -\frac{n\log(1/p)}{(2+\epsilon)\log n}.
\]
\end{theorem}

The improvements we make here will all be for
the range $p\ge1/2$.  Accordingly, we extract the following
upper and lower bounds from Theorems~\ref{thm:old-ub-general} and~\ref{thm:old-lb}
for this range. Note that the modest improvement in the upper bound results from 
the narrowing of the range on $p$.  Note also that $p$ is bounded away from $1$ for 
the lower bound.  

\begin{corollary}[Old Upper and Lower Bounds]\label{cor:old-bounds}
Suppose $1/2\le p<1$, then $a.a.s.$, 
\begin{equation*}
\dim(P) < n-\frac{n\log(1/p)}{2\log n}.
\end{equation*}
Furthermore, for every $\epsilon>0$, there exists $\delta>0$ so that
$p<1-n^{-1+\epsilon}$, then $a.a.s.$,
\begin{equation*}
\dim(P)>n -\frac{\delta n}{\log n}.
\end{equation*}
\end{corollary}

In discussing the merits of our new results, the old upper and lower bounds given in 
Corollary~\ref{cor:old-bounds} will be the basis of comparison.  

To make the connection with Question~1 for posets concrete, let
$p=1-n^{-1/2}$.  If $P\in\Omega(n,p)$, then $a.a.s.$,
$\se(P)=O(\sqrt{n}\log n)$. Furthermore, the old upper bound in
Corollary~\ref{cor:old-bounds} 
implies that $a.a.s.$, $\dim(P)\le n- \sqrt{n}/(2\log n)$.  It is easy to
see that if this 
upper bound is tight, up to a poly-log multiplicative factor on 
the difference $n-\dim(P)$, then the exponent~$2$ on $c$ in the 
function $\sa(c)$ is correct.

However, the old lower bound in Corollary~\ref{cor:old-bounds} only 
asserts that there is a constant $\delta$ such that 
$a.a.s.$, $\dim(P) \ge n -\delta n/\log n$.  This inequality is enough to
prove that $\sa(c)=\Omega(c\log c)$, but we already had a constructive 
proof of an even better lower bound.  This shortcoming was the launching point 
for revisiting the subject of dimension for random bipartite posets, but
now with the specific goal of obtaining better bounds when $p\ge1/2$.

The bridges to Question~2 for posets were not clear at the outset of our
research but came into view as better bounds and connections to Question~1 unfolded.
These details will become clear later in the paper.

\subsection{Statement of Improved Bounds for Random Bipartite Posets}

To avoid sporadic effects when $p$ is very close to $1$, we assume $p\le 1-n^{-1}\log^2 n$.  
Consistent with modern research in combinatorics, we will typically
treat a quantity like $(1-p)n/\log n$ as if it is an integer when
it tends to infinity with $n$.  The minor errors this approach introduces can 
be easily repaired.

For the range $1/2\le p\le 1-n^{-1}\log^2 n$, many of the results 
and arguments are more naturally phrased in terms of the complementary parameter $q=1-p$.
Accordingly, for the balance of the paper, the symbol $p$ will be used exclusively
as a quantity (usually a function of $n$) from the interval $[0,1]$, while $q$ will
\emph{always} be $1-p$.  Some of the proofs of our new results are extensions and 
refinements of arguments appearing in~\cite{ErKiTr91} and~\cite{BHPT16}, but
most of our results require entirely new approaches.  In particular, we will
apply some second moment methods, Talagrand's inequality and Janson's inequality.
None of these tools were used in~\cite{ErKiTr91} or in~\cite{BHPT16}.  

Our improved upper bounds, stated below in comprehensive form, 
involve the well-studied Euler product function $\phi(q)=\prod_{i=1}^{\infty}(1-q^i)$.  

\begin{theorem}\label{thm:new-ub-comp}
Suppose $n^{-1}\log^2 n\le q\le 1/2$, $0<\epsilon<1$ and 
$z=n^2q\log(1/\phi(q))$.  Then $a.a.s.$,
\begin{equation*}\label{eqn:new-ub-comp}
\dim(P) < 
  \begin{cases}
  n -(2-\epsilon)\log(qn)/q & \text{if } n^{-1}\log^2 n\le q\le n^{-1/2}\log n.\\
  n -qn/(2\log(qn)) &\text{if } n^{-1/2}\log n <q\le n^{-1/3}.\\
  n -n\log(1/\phi(q))/\log z&\text{if } n^{-1/3}<q\le 1/2.
  \end{cases}
\end{equation*}
\end{theorem}

In the discussion to follow, we will refer to these three inequalities
as New Upper Bounds~(1), (2), and~(3).   New Upper Bounds~(2) and~(3) are 
minor improvements relative to the old upper bound. For example, 
when $q=1/2$, the old upper bound asserts that $a.a.s.$, 
$\dim(P)< n-0.346n/\log n$.
On the other hand, New Upper Bound~(3) improves this to $a.a.s.$, 
$\dim(P)< n - 0.621 n/\log n$.  When $q=n^{-1/2}\log n$, the old upper bound 
asserts that $a.a.s.$, $\dim(P)<n - \sqrt{n}/2$, while New Upper Bound~(2)  
improves this to $a.a.s.$, $\dim(P)<n - \sqrt{n}$.  
 
New Upper Bound~(1) is substantially better 
than the old bound.  For example, when $q=n^{-\alpha}$
and $0<\alpha<1/2$, the old upper bound asserts that $a.a.s.$,
$\dim(P)<n-n^{1-\alpha}/2\log n$, while New Upper Bound~(1) improves this
to $a.a.s.$, $\dim(P)<n-(1-\alpha)n^{\alpha}\log n$.

Here is a statement, again in comprehensive form, of our improved lower bounds.  

\begin{theorem}\label{thm:new-lb-comp}
Suppose $n^{-1}\log^2n\le q\le 1/2$ and $0<\epsilon<1$.
If $q\ge n^{-1/4}\log^3 n$, set $z=\log n+4\log q - 8\log\log n$.
Then $a.a.s.$,
\begin{equation*}
\dim(P) >
  \begin{cases} 
    n-(2+\epsilon)\log(qn)/q &\text{if } n^{-1}\log^2 n \le q \le n^{-4/5}.\\
    n-32\big(n\log n/q\bigr)^{1/2} &\text{if }n^{-4/5}\le q\le (32)^{1/3}n^{-1/3}\log^{1/3} n.\\
    n-8qn &\text{if } (32)^{1/3}n^{-1/3}\log^{1/3} n\le q\le \frac{1}{8} n^{-1/4}\log^3 n.\\
    n-24qn/z &\text{if }\frac{1}{8} n^{-1/4}\log^3 n\le q\le 1/2.
   \end{cases}
\end{equation*}
\end{theorem}

In the discussion to follow, we will refer to these inequalities
as New Lower Bounds~(1) through (4).  In the range
$n^{-1}\log^2 n\le q\le 1/2$, our bounds show that $a.a.s.$, $\dim(P)\ge(1-o(1))n$.
So the accuracy of our upper and lower bounds should be judged on
the quantity $n-\dim(P)$.  For this quantity,  our upper and lower bounds differ by 
a multiplicative factor that is $1+o(1)$ when $n^{-1}\log^2 n \le q\le n^{-4/5}$, 
and by $O(\log n)$ when $n^{-1/3}\le q\le 1/2$.  In~\cite{ErKiTr91}, it was
asked whether the expected value of $\dim(P)$ behaves monotonically as a function
of $p$.  Our results answer this question negatively.  In particular, when
$q=n^{-1/4}$, we have $a.a.s.$, $\dim(P)\ge n-8n^{3/4}$ and when $q=n^{-4/5}$,
we have $a.a.s.$, $\dim(P)\le n- n^{4/5}$.

In the range $n^{-4/5}\le q\le n^{-1/3}$, 
the ratio of our two bounds on $n-\dim(P)$ deteriorates.  In particular, for 
the special value of $q=n^{-1/2}$, we are only able to show that (roughly speaking) 
$n^{1/2}\le n-\dim(P) \le n^{3/4}$.  

The remainder of the paper is organized as follows.  In the next section,
we provide essential background material.  
In the following three sections, we give proofs of our 
new bounds, grouping arguments according to the underlying scheme.  The setup for
the application of the inequalities of Talagrand and Janson will be given
just before the results are needed.  We return to the motivating 
extremal problems in Section~\ref{sec:eps}, and we close with
some brief comments on remaining problems in Section~\ref{sec:close}.

\section{Essential Background Material}\label{sec:background}

For a poset $P$, we use the compact notation $a<_P a'$ when 
$a<a'$ in $P$. Similarly, we write $a\parallel_P a'$ when $a$ is incomparable 
to $a'$ in $P$.  However, to avoid double subscripts, when $L_j$ is a 
linear extension of $P$, we will use the long form $a<a'$ in $L_j$.

We will be concerned (almost exclusively) with the class $\bbB$
of \textit{bipartite} posets whose ground set is the union of two disjoint 
antichains $A$ and $A'$ with $A\subseteq \Min(P)$ and $A'\subseteq\Max(P)$.  
Bipartite posets have been studied extensively in the literature, and we 
will follow here the conventions that have emerged in this research.  
We will write $\bbB(n)$ for the class of bipartite posets with $|A|=|A'|=n$.

For a poset $P\in\bbB$, we let $I_P$ consist of all pairs
$(a,a')\in A\times A'$ with $a\parallel_P a'$.  Let $d\in N$, 
and let $\cgF=\{L_1,\dots,L_d\}$ be a family of linear extensions of $P$.
We abuse notation slightly and call $\cgF$ a \textit{realizer of $P$} if
for every $(a,a')\in I_P$, there is some $j\in[d]$ with $a>a'$ in $L_j$.
We then define the \textit{dimension} of $P$, denoted
$\dim(P)$, as the least positive integer $d$ such that $P$ has a realizer
of size~$d$.  It is easy to see that this altered notion of dimension never 
exceeds the original Dushnik-Miller definition, and the difference is at most~$1$.
In our work, an additive error of this magnitude can be safely ignored.

\subsection{Matchings, Independence Number, and Clique Number}

The following proposition, which holds for posets in general, is nearly 
self-evident.  It is stated for emphasis.

\begin{proposition}\label{pro:Hira}
Suppose that $(x,y)$ is an ordered pair of distinct
points in a poset $P$ with $x\parallel_P y$.  Then
there exists a linear extension $L=L(x,y)$ of $P$ such that:
\begin{enumerate}
\item If $w\in P$ and $w\parallel_P x$, then $x>_L w$.
\item If $z\in P$ and $z\parallel_P y$, then $z>_L y$.
\end{enumerate}
\end{proposition}

When $P\in\bbB$ and $(a,a')\in A\times A'$ with $a\parallel_P a'$, 
we let $\cgL(a,a')$ denote the set of linear extensions of $P$ satisfying 
the requirements of Proposition~\ref{pro:Hira} for the pair $(a,a')$.

Let $P\in\bbB$ and let $d\in[n]$.  A \textit{matching} (of size~$d$) in $P$
consists of a pair $(T,T')$ of $d$-element subsets of $A$ and $A'$, respectively, and
labelings $T=\{a_1,\dots,a_d\}$ and $T'=\{a'_1,\dots,a'_d\}$ such that 
$a_j\parallel_P a'_j$ for every $j\in[d]$.  There are obvious notions of 
\textit{maximal} and \textit{maximum} matchings.  
Also, when $T\subseteq A$, $T'\subseteq A'$, and $d=|T|=|T'|$, we say that $T$ and
$T'$ \textit{can be matched} if such labelings of $T$ and $T'$ exist.

\begin{lemma}\label{lem:matchings-1}
Let $P\in\bbB$.  Then $dim(P)$ is at most the minimum
size of a maximal matching in $P$.
\end{lemma}

\begin{proof}
Let $T=\{a_1,\dots,a_d\}$ and $T'=\{a'_1,\dots,a'_d\}$ be labelings determining
a matching $(T,T')$ of size $d$ in $P$.  If this matching is maximal, we show that
$\dim(P)\le d$.  For each $j\in[d]$, let $L_j$ be any linear extension 
in $\cgL(a_j,a'_j)$.  Clearly, $\cgF=\{L_1,\dots,L_d\}$ is a realizer of $P$.
\end{proof}

The following elementary lemma is implicit in~\cite{ErKiTr91} and
explicit in~\cite{BHPT16}.

\begin{lemma}\label{lem:matchings-2}
Let $P\in\bbB$.   If $I_P\neq \emptyset$, 
and $\dim(P)=d$, then there is a realizer $\cgF=\{L_1,\dots,L_d\}$ of
$P$ for which there is a matching $(T,T')$ with $T=\{a_1,\dots,a_d\}$ and
$T'=\{a'_1,\dots,a'_d\}$ in $P$ such that $L_j\in\cgL(a_j,a'_j)$ for
each $j\in[d]$.
\end{lemma}

\begin{proof}
Given a realizer $\cgF=\{L_1,\dots,L_d\}$, we note that for every $j\in[d]$, the 
highest element of $A$ is over the lowest element of $A'$.  If this assertion
failed for some $j\in[d]$, then $\cgF-\{L_j\}$ would be a realizer for $P$. 
Now  carry out the following
modifications, in an iterative manner, to the linear extensions in $\cgF$.   
For each $j=\in[d]$, let $a_j$ be the highest element of $A$.  Since $d=\dim(P)$,
loss of generality, we may assume that $a_j>a'$ in $L_j$ for every $a'\in A'$ with
$a_j\parallel_P a'$.  If not, simply move all such elements from above
$a_j$ to the gap immediately below it.  An analogous remark holds
for the lowest element $a'_j$ of $A'$ in $L_j$.  Then for all $k$ with $j<k\le d$,
move $a_j$ to the bottom of $L_k$ and move $a'_j$ to the top of $L_k$.
After these steps have been taken, the resulting family is a realizer 
satisfying the requirements of the lemma.
\end{proof}

Throughout this paper, we will exploit connections between posets and graphs, and
we have already discussed the \textit{clique number} of a graph $G$, denoted $\omega(G)$.  
Here is the analogous concept for bipartite posets.  Let $P\in\bbB$.  We call a 
pair $(V,V')$ a \textit{clique pair} when
$V\subseteq A$, $V'\subseteq A'$, and $v <_P v'$ for all
$(v,v')\in V\times V'$.  A clique pair $(V,V')$ is \textit{balanced} if
$|V|=|V')$.  In turn, we define the \textit{balanced clique number} of $P$,
denoted $\wttbcn(P)$, 
as the largest integer $r$ such that $P$ contains a clique pair $(V,V')$ with
$|V|=|V'|=r$. By convention, $\wttbcn(P)=0$ if there
is no pair $(a,a')\in A\times A$ with $a<_P a'$.

Let $G$ be a graph.  A set $I$ of vertices in $G$ is said to be \textit{independent}
if there are no edges in $G$ with both endpoints in $I$.  In turn, the
\textit{independence number} of $G$, denoted $\alpha(G)$, is the
maximum size of an independent set of vertices in $G$.   
Analogously, when $P\in \bbB$, we will refer to a pair $(U,U')$ as an 
\textit{independent pair}, when $U\subseteq A$, $U'\subseteq A'$, and $u\parallel_P u'$ for all
$(u,u')\in U\times U'$.  An independent pair $(U,U')$ is \textit{balanced}
if $|U|=|U'|$, and the \textit{balanced independence number of $P$}, denoted 
$\wttbin(P)$, is the largest
integer $s$ such that $P$ contains an independent pair $(U,U')$ with $|U|=|U'|=s$.
Now $\wttbin(P)=0$ if there is no pair 
$(a,a')\in A\times A'$ with $a\parallel_P a'$.  

The following lemma is implicit in~\cite{BHPT16}.

\begin{lemma}\label{lem:2-mixed}
Let $P\in\bbB$ and suppose that $\wttbin(P)<2$.
If $I_P\neq\emptyset$, then $\dim(P)$ is the minimum size of a maximal matching in $P$.
\end{lemma}

\begin{proof}
Let $d=\dim(P)$.  We know from Lemma~\ref{lem:matchings-2} that
$d$ is at most the minimum size of a maximal matching.  We now show that
this inequality is tight.  Let $\cgF=\{L_1,\dots,L_d\}$ be a realizer
of $P$ satisfying the requirements of Lemma~\ref{lem:matchings-2}.
Then let $T=\{a_1,\dots,a_d\}$ and $T'=\{a'_1,\dots,a'_d\}$ be the
matching associated with $\cgF$. Set $M=A-T$ and $M'=A'-T'$.

We claim that the matching evidenced by $T$ and $T'$ is maximal.  
Suppose this assertion fails and there is a pair $(x,x')\in M\times M'$ 
with $x\parallel_P x'$.  Since $\cgF$ is a realizer, there is some 
$j\in[d]$ with $x>x'$ in $L_j$.  This implies that both elements
of $U=\{a_j,x\}$ are incomparable with both elements of $U'=\{a'_j,x'\}$.
In turn, this implies that $(U,U')$ is a balanced independent pair in $P$, so that
$\wttbin(P)\ge2$.  The contradiction completes the proof.
\end{proof} 

\section{Matchings, Clique Number,  Independence Number and Talagrand's Inequality}

In this section, we prove New Upper Bound~(1) and New Lower Bound~(1).
The proofs have the same flavor, and where their ranges overlap, we
are able to determine $a.a.s.$, the expected value of $n-\dim(P)$ to
within a multiplicative ratio that is at most $1+o(1)$.

The arguments for these bounds require preliminary lemmas, some of
which may be of independent interest.
Let $s\in [n]$.  Also let $S$ and $S'$ be $s$-element subsets of
$A$ and $A'$, respectively, with $s=|S|=|S'|$.  We say the \textit{defect}
of the pair $(S,S')$ is $s$ if $(S,S')$ is a clique; otherwise, the defect
of $(S,S')$ is the least non-negative integer $\delta$ such that
there are subsets $T\subseteq S$ and $T'\subseteq S'$ with
$s-\delta=|T|=|T'|$ such that $T$ and $T'$ can be matched.

\begin{lemma}\label{lem:matchings-defect}
Suppose that $n^{-1}\log^2 n \le q\le n^{-1/2}\log n$.  Then $a.a.s.$, the following
statement holds:\quad
If $n/2\le s\le n$, $S\subseteq A$, $S'\subseteq A'$, and $s=|S|=|S'|$, then 
the defect of the pair $(S,S')$ is at most $24/q$.
\end{lemma}

\begin{proof}
Set $\delta=24/q$.
There are at most $2^{2n}$ pairs of the form $(S,S')$ where
$S\subseteq A$, $S'\subseteq A'$ and $|S|=|S'|\ge n/2$.  The lemma follows
if we can show that for any such pair, the probability that
there is no matching of size $|S|-\delta$ between $S$ and $S'$ is
$o(2^{-2n})$.  Fix such a pair and let $s=|S|=|S'|$.

For each non-empty subset $W\subseteq S$, let $N(W)$ consist of all elements
of $S'$ that are incomparable with at least one element of $W$.  If
$|N(W)|\ge |W|-\delta$, for all subsets $W\subseteq S$, then a matching
of the desired size exists.  So we consider the event $F$ that holds
if there is some subset $W\subseteq S$ with $|N(W)|<|W|-\delta$.

Trivially, the inequality $|N(W)|\ge |W|-\delta$ holds when
$|W|\le \delta$.  Also, if $F$ fails for all sets $W\subset A$ with
$|W|=s-\delta$, then it fails for all sets $W\subset A$ with
$|W|>s-\delta$.  It follows that $F\subseteq\cup_{i=\delta}^{s-\delta}
F_i$, where event $F_i$ holds when there is an $i$-element subset
$W\subseteq S$ such that $|N(W)|<|W|-\delta$.

Now suppose that $\delta\le i\le s/2$.  Then there are at most $\binom{s}{i}$ choices
for the set $W$.  For each choice of $W$, there are at most $\binom{s}{i}$ choices
for an $i$-element subset $W'$ of $S'$ such that $W'\cap N(W)=\emptyset$.  It follows
that
\begin{align*}
\bbP(F_i)&\le \binom{s}{i}^2 (1-q)^{i(s-i)}\\
        &< \exp(2i\log n)(1-q)^{in/4}              &&\text{since $s-i\ge n/4.$}\\
        &< \exp(2i\log n)\exp(-iqn/4)\\
        &< \exp(-iqn/8)                           &&\text{since $qn/4> 4\log n$.}\\
        &< \exp(-\delta qn/8)                      &&\text{since $i\ge \delta$.}\\
        &= \exp(-3n)                               &&\text{substituting for $\delta$.}\\
        &=o(2^{-2n}/n).
\end{align*}

A symmetric calculation shows $\bbP(F_i)=o(2^{-2n}/n)$ when
$s/2\le i\le s-\delta$.  It follows that
\[
\bbP(F)\le\sum_{i=\delta}^{s-\delta}= o(2^{-2n}).
\]
This completes the proof of the lemma.
\end{proof}

For the proof of the next lemma, we follow (essentially) the notation 
and terminology of Corollaries~4.3.3, 4.3.4 and~4.3.5 in Alon and 
Spencer~\cite{AloSpe16}.  For a random variable $X$, we denote the
expected value of $X$ as $\bbE[X]$.  

\begin{lemma}\label{lem:t=2}
Suppose $n^{-1}\log^2 n\le q\le n^{-1/2}\log n$.  Let
$X$ be the random variable counting the number of balanced independent pairs of
size~$2$. Then $\bbE[X]\rightarrow\infty$ and $a.a.s.$,  $X\sim \bbE[X]$.
\end{lemma}

\begin{proof}
Set $m=\binom{n}{2}^2$.  Then let $\{(U_1,U'_1),\dots,(U_m,U'_m)\}$ be
a listing of pairs such that for each $i\in[m]$, $U_i$ and $U'_i$ are
$2$-element subsets of $A$ and $A'$, respectively.  For each $i\in[m]$,
we have an event $E_i$ that holds if $(U_i,U'_i)$ is a balanced
independent pair. Also, we let $X_i$ be the associated indicator random 
variable.  Then $X=X_1+\dots+X_m$, and we note that the random variables 
$X_1,\dots,X_m$ are symmetric.

We note that  since $q\ge n^{-1}\log^2 n$,
\begin{align*}
\bbE[X]&=\binom{n}{2}^2q^4
       \ge n^4q^4/5
       \ge \log^8 n/ 5
       \rightarrow\infty.
\end{align*}

We write $E_i\sim E_j$ when $i,j$ are distinct elements of $[m]$ and
the events $E_i$ and $E_j$ are dependent.  Clearly, $E_i\sim E_j$ when
$U_i\times U'_i$ and $U_j\times U'_j$ intersect.  When they intersect, the number
of common pairs is either~$1$ or $2$.  We fix an index
$i$ and then calculate the quantity $\Delta^*$ defined by
\begin{equation}\label{eqn:Delta-star-1}
\Delta^*=\sum_{j\sim i}\bbP[E_j|E_i].
\end{equation}
There are $4(n-2)^2$ choices for the index $j$ so that 
$U_i\times U'_i$ and $U_j\times U_j$ have exactly one common pair.  
For each such $j$, the value of $\bbP[E_j|E_i]$ is
$q^3$.  Similarly, there are $4(n-2)$ choices for the index $j$ so that
$U_i\times U'_i$ and $U_j\times U'_j$ have exactly two common pairs.  For each such
$j$, the value of $\bbP[E_j|E_i]$ is $q^2$.

Using first that $q\ge n^{-1}\log^2 n$ and then that $q\le n^{-1/2}\log n$, we have
\begin{align*}
\Delta^*&=4q^3(n-2)^2+4q^2(n-2)
        < 4q^5n^4/\log^4 n+4q^5n^4/\log^ 6 n
        =o(\bbE[X]).     
\end{align*}
Now the conditions of Corollary~4.3.5 from~\cite{AloSpe16} are satisfied and we
conclude that almost always, $X\sim \bbE[X]$.
\end{proof}

The next lemma is a straightforward application of Markov's inequality.

\begin{lemma}\label{lem:markov-bcn}
Suppose $n^{-1}\log^2 n\le q\le n^{-1/2}\log n$ and $0<\epsilon<1$. Then $a.a.s.$,
\[
\wttbcn(P) < (2+\epsilon)\log(qn)/q.
\]
\end{lemma}

\begin{proof}
Set $r=(2+\epsilon)\log(qn)/q$ and let $X$ count the number of balanced
clique pairs of size $r$ in $P$.  Then
\begin{align*}
\bbE[X]&=\binom{n}{r}^2 (1-q)^{r^2}
       <n^{2r}\exp(-qr^2)
       =\exp(-2\epsilon\log(qn))
       =o(1).
\end{align*}
Since $\bbE[X]=o(1)$, it follows that $a.a.s.$, $\wttbcn(P) < r$.
\end{proof}
     
The  elementary inequality in Lemma~\ref{lem:markov-bcn} is essentially best possible.
However, this assertion is \textit{considerably} more challenging to prove\footnote{
We are grateful to Tomasz \L uczak who \textit{greatly} assisted us in this effort,
especially the use of Talagrand's Inequality, and
the setup using a vertex martingale.}.

There are several different forms of Talagrand's inequalities in the
literature.  We will use the version given  in~\cite[Theorem~2.29]{JaLuRu00}. 
Let $\bfR$ and $\bfR^+$ denote, respectively, the set of real numbers and the set of
positive real numbers.
When $(\Lambda_1,\dots,\Lambda_n)$ is a sequence of subsets of $\bfR$,
we denote by $\Lambda$ the product $\Lambda_1\times\dots\times \Lambda_n$.  When
$i\in[n]$, $\overline{z}\in\Lambda$, we denote by $\overline{z}(i)$ the value of 
coordinate $i$ of $\overline{z}$. 


\begin{theorem}[Talagrand Inequality]\label{thm:Talagrand}
Let  $X=f(Z_{1},\dots,Z_{n})$ be a random variable determined by $n$ 
independent trials $Z_1,\dots,Z_n$, where $f:\Lambda \to \bfR$ and each 
$Z_{i}$  takes on values in a finite set  $\Lambda_{i}$. Suppose 
$c_1,\dots,c_n\in\bfR^{+}$
and $\psi:\bfR\to \bfR$.
If for all $\overline{z},\overline{w}\in\Lambda$, both
\begin{enumerate}
\item for all $i\in[n]$, if $\overline{z}(j)=\overline{w}(j)$ for all 
$j\in [n]-\{i\}$, then $|f(\overline{z})-f(\overline{w})|\le c_i$, and
\item for all $\alpha\in \bbR$, if $f(\overline{z})\ge\alpha$, then
there is $J\subseteq [n]$ such that both 
\begin{enumerate}
\item $\sum_{j\in J}c_j^2\le\psi(\alpha)$,  and
\item if $\overline{z}(j)=\overline{w}(j)$ for all
$j\in J$, then $f(\overline{w})\ge\alpha$, 
\end{enumerate}
\end{enumerate}
then for every $\gamma\in\bbR$ and every $\beta\ge 0$,
\begin{equation*}
\bbP(X\le \gamma-\beta)\bbP(X\ge \gamma)\le e^{-\frac{\beta^2}{4\psi(\gamma)}}. 
\end{equation*}
\end{theorem}


\begin{lemma}\label{lem:Luczak}
If $n^{-1}\log^2n\le q\le n^{-1/2}\log n$ and $0<\epsilon<1$, then $a.a.s.$,
$\wttbcn(P)\ge (2-\epsilon)\log(qn)/q$.
\end{lemma}

\begin{proof}
Before we begin calculations, we explain how Theorem~\ref{thm:Talagrand} will be applied.
Label the elements of $A$ and $A'$
arbitrarily (no assumptions about matchings) as $A=\{a_1,\dots,a_n\}$ and $
A'=\{a'_1,\dots,a'_n\}$.  For each $i\in[n]$, let $\Lambda_i$ be the family
of all subsets of $A'$.  Then it is natural to view $\Lambda$ as just
a coding of the posets in $\bbB_n$, i.e., a poset $P\in\bbB(n)$ uniquely
determines for each $i\in[n]$ the set of all $a'\in A'$ with
$a_i<_P a'$. In turn, the random variables $Z_i$ with $i\in[n]$ then
capture the space $\Omega(n,p)$.

We define a function $f:\Lambda\rightarrow N$ by
setting $f(P)=\wttbcn(P)$.  Then we have a random variable $X=f(P)$.  We note 
that if $P$ and $Q$ are posets in $\bbB(n)$, and there
is some $a_i\in A$ such that the only differences between $P$ and
$Q$ involve pairs from $\{a_i\}\times A'$, 
then $|f(P)-f(Q)|\le 1$, i.e., we take
$c_i=1$ for all $i\in[n]$.  Furthermore, whenever we have
$f(P)\ge \alpha$, this can be certified by a set $J$ of size $\alpha$.
So we simply take $\psi(\alpha)=\alpha$.

With $\epsilon$ fixed, we want to show that 
$\bbP\bigl(X >(2-\epsilon)\log(qn)/q\bigr)$ tends to $1$.  Set
$\delta=\epsilon/2$.  It is enough to prove that
$\bbP\bigl(X >(2-\epsilon)\log(qn)/q\bigr)\ge 1-\delta$.
Set $k=\gamma=(2-\delta)\log(qn)/q$ and $\beta=\delta \log(qn)/q$.  Note
that $\gamma-\beta=(2-\epsilon)\log(qn)/q$.  Substituting
these values into Talagrand's inequality,  we obtain:
\begin{equation*}
\bbP\bigl(X\le(2-\epsilon)\log(qn)/q\bigr)\bbP\bigl(X\ge k\bigr)\le e^{-\frac{\delta^2\log^2(qn)}{4q^2k}}.
\end{equation*}
Substituting for $k$ in the right hand side of the last inequality, we obtain
\begin{align*}
\bbP\bigl(X\le(2-\epsilon)\log(qn)/q\bigr)\bbP\bigl(X\ge k\bigr)&\le e^{-ck}&\text{where }
c= \frac{\delta^2}{4(2-\delta)^2}.
\end{align*}

This implies that either (1)~$\bbP\bigl(X\le (2-\epsilon)\log(qn)/q\bigr)\le\delta$ or
(2)~$\bbP(X\ge k)\le e^{-ck}/\delta$.  To complete the proof, we need only show
that statement~(2) cannot hold.  This will be accomplished by showing that
\begin{equation}\label{eqn:not-so-fast}
\bbP(X\ge k) \ge e^{-o(k)}.
\end{equation}

Now let $Y$ be the random variable counting the number of balanced
clique pairs $(V,V')$ of size~$k$.  Then $X\ge k$ if and only if $Y>0$.
Then from~\cite[Remark~3.1]{JaLuRu00}, we have
\begin{equation}\label{eqn:remark-bound}
\bbP(Y>0)\ge \frac{(\bbE[Y])^2}{\bbE[Y^2]}.
\end{equation}

Accordingly, we want to show that $\bbP(Y>0)\ge e^{-o(k)}$.
Working with the reciprocal, which simplifies the analysis,
we then want to show that:
\begin{equation}\label{eqn:Luczak-up}
\frac{\bbE[Y^2]}{(\bbE[Y])^2}\le e^{o(k)}.
\end{equation}
We have:
\begin{align*}
\frac{\bbE[Y^2]}{(\bbE[Y])^2}&=\frac{\bigl(\bbE[Y]\bigr)\sum_{i=0}^k\sum_{j=0}^k \binom{k}{i}
  \binom {n-k}{k-i}\binom{k}{j}\binom{n-k}{k-j}(1-q)^{k^2-ij}}{(\bbE[Y])^2}\\
&=\sum_{i=0}^k\sum_{j=0}^k \frac{\binom{k}{i}\binom{n-k}{k-i}\binom{k}{j}\binom{n-k}{k-j}}
 {\binom{n}{k}^2}(1-q)^{-ij}.
\end{align*}

There are $(k+1)^2$ terms in the sum, so it suffices to show that every term has
size at most $\exp(o(k))$.  Clearly, this holds whenever $i=0$ or $j=0$.
So we are concerned only with terms where $i,j\ge1$.  Trivially, we
have $\binom{k}{i}\binom{n-k}{k-i}\le\binom{n}{k}$, so that
\begin{equation}\label{eqn:easy}
\frac{\binom{k}{i}\binom{n-k}{k-i}}{\binom{n}{k}}\le 1.
\end{equation}
A symmetric inequality holds for $j$.  

Using only the elementary bound in~(\ref{eqn:easy}), we observe that if 
$i\le k/\log^2(qn)$, then the term for the pair $(i,j)$ is at most 
\begin{align*}
(1-q)^{-ij}&\le \exp(qk^2/\log^2(qn))\\
           &\le \exp(qk\frac{2\log(qn)}{q}/\log^2(qn))\\
           &=\exp(2k/\log(qn))\\
           &=\exp(o(k)).
\end{align*}
A symmetric statement holds when $j\le k/\log^2(qn)$.

Now we focus on the terms when $i,j\ge k/\log^2(qn)$. For such terms,
we have the following improved bound.
\begin{align*}
\frac{\binom{k}{i}\binom{n-k}{k-i}}{\binom{n}{k}}<
  \binom{k}{i}\frac{\binom{n}{k-i}}{\binom{n}{k}}
  <\Biggl(\frac{3k}{i}\Biggr)^i\Biggl(\frac{k}{n}\Biggr)^i
  =\Biggl(\frac{3k^2}{in}\Biggr)^i.
\end{align*}
A symmetric inequality holds for $j$.    Accordingly, when $k/\log^2(qn)\le i,j\le k$,
we have the following upper bound on the term for $(i,j)$:
\begin{equation}\label{eqn:gen-ub-ij}
\Biggl(\frac{3k\log^2(qn)}{n}\Biggr)^{i+j}\exp(qij).
\end{equation}

We assume without loss of generality that $i\le j$.  We then take the logarithm of 
the expression in~(\ref{eqn:gen-ub-ij}) to obtain:
\begin{equation}\label{eqn:f(i)}
(i+j)\bigl[\log3+\log k +2\log\log(qn)-\log n\bigr]+qij
\end{equation}
Considering $j$ fixed, this is a linear function of $i$, defined on the
interval $[k/\log^2(qn), j]$.  So it achieves its maximum
value either at $i=k/\log^2(qn)$ or at $i=j$.  The choice depends on the sign of the
coefficient of $i$,  which is
\begin{multline}\label{eqn:coefficient-of-i}
\log 3+\log k +2\log\log(qn)-\log n + qj=\\
\log 3+\log(2-\delta)+ 3\log\log(qn) - \log(qn)+qj.
\end{multline}

\medskip
\noindent
\textbf{Case 1.}\quad
$j\le \bigl(\log(qn)-3\log\log(qn)-\log 3-\log(2-\delta)\bigr)/q$.

\smallskip
In this case, the coefficient of $i$ is negative, so the maximum value is
achieved when $i=k/\log^2(qn)$.   The term associated with $(i,j)$ is
less than $\exp(qij)$, and for $j$, we use the
generous upper bound bound $j\le \log(qn)/q$.   It follows that the term for
$(i,j)$ is at most:
\[
\exp(q\frac{k}{\log^2(qn)}\frac{\log(qn)}{q}=\exp(k/\log(qn)=\exp(o(k)).
\]

\medskip
\noindent
\textbf{Case 2.}\quad
$j> \bigl(\log(qn)-3\log\log(qn)-\log3-\log(2-\delta)\bigr)/2$.

In this case, the maximum value is achieved when $i=j$, and the term
associated with $(i,j)$ is at most:
\begin{equation}\label{eqn:case2-1}
\Biggl(\Bigl(\frac{9k^2\log^4(qn)}{n^2}\Bigr)^2\exp(qj)\Bigg)^j
\end{equation}
In~(\ref{eqn:case2-1}), we note that
\[
\exp(qj)\le \exp(qk)=(qn)^{2-\delta}.
\]
Using this inequality and substituting for $k$, the expression in~(\ref{eqn:case2-1})
becomes:
\[
\Biggl(\frac{9\log^6(qn)}{(qn)^\delta}\Biggr)=o(1).
\]
With this observation, the proof of the lemma is complete.
\end{proof}

\begin{lemma}[New Upper Bound (1)]
Suppose $n^{-1}\log^2 n\le q\le n^{-1/2}\log n$ and $0<\epsilon<1$.
Then $a.a.s.$, $\dim(P) \le n - (2-\epsilon)\log(qn)/q$.
\end{lemma}

\begin{proof}
Let $\epsilon_1=\epsilon/2$.  Then set $r=(2-\epsilon_1)\log(qn)/q$.
Using Lemma~\ref{lem:Luczak}, it follows that $a.a.s.$, the following two statements
hold: (1)~the balanced clique number of $P$ is
at least $r$, and~(2)~for every pair $(S,S')$, with $S\subseteq A$,
$S'\subseteq A'$ and $|S|=|S'|\ge n/2$, the defect of $(S,S')$ is less than 
$24/q$.

Set $\delta=24/q$, and let $(V,V')$ be a balanced clique pair of size~$r$ in $P$.  
Set $S=A-V$ and $S'=A'-V'$.  Note that $|S|=|S|> n/2$.  Let $(T,T')$ be a 
maximum matching in $S\cup S'$.  Then $|T|=|T'|\ge |S|-\delta$.
Let $Q$ be the subposet of $P$ determined by the points in the matching together
with the points in $V\cup V'$.  Then $(T,T')$ is a maximal matching in $Q$.  It follows
from Lemma~\ref{lem:matchings-1} that $\dim(Q)\le|T| \le n-r$.
We note that
\[
2\delta < 48/q <\epsilon_1\log(qn)/q.
\]
The removal of a point from a bipartite poset decreases dimension by at most~$1$,
and $Q$ is obtained from $P$ by removing at most $2\delta$ points.
It follows that $a.a.s.$,
\begin{align*}
\dim(P)&\le\dim(Q)+2\delta
       \le (n-r)+\epsilon_1 \log(qn)/q
       = n-(2-\epsilon)\log(qn)/q.\qedhere
\end{align*}
\end{proof}

Readers will note that the proof of the next result uses the ``alteration''
method (see Chapter~3 in Alon and Spencer~\cite{AloSpe16}) first used
by Erd\H{o}s in his probabilistic proof of the existence of graphs
with large girth and large chromatic number.

\begin{lemma}[New Lower Bound~(1)]
Suppose $0<\epsilon<1$.  If $n^{-1}\log^2 n\le q\le n^{-4/5}$, then $a.a.s.$,
$\dim(P)> n-(2+\epsilon)\log(qn)/q$.
\end{lemma}

\begin{proof}
Set $\epsilon_1=\epsilon/2$ and $r=(2+\epsilon_1)\log(qn)/q$.
Using Lemma~\ref{lem:markov-bcn}, we know that $a.a.s.$,
the balanced clique number of $P$ is less than $r$.

Let $X$ be the random variable counting the number of balanced independent pairs $(U,U')$  
with $|U|=|U'|=2$.  Using Lemma~\ref{lem:t=2}, we know that
$a.a.s.$, $X\sim E[X]$.  Since $E[X]=\binom{n}{2}^2q^4\sim q^4n^4/4$, we
will settle for the weaker inequality $a.a.s.$, $X\le n^4q^4/2$.  
When $q\le n^{-4/5}$, this implies that $a.a.s.$, $X < \epsilon_1\log(qn)/q$.

It follows that there are subsets $S\subset A$ and $S'\subset A'$ with 
$|S|=|S'|=n-\epsilon_1 q^{-1}\log(qn)$ such that $a.a.s.$, the subposet $Q$ of $P$ 
with ground set $S\cup S'$ has balanced clique number less than~$r$ and balanced
independence number less than~$2$  Then $a.a.s.$,
\begin{align*}
\dim(P)&\ge \dim(Q)\\
       &>\bigl(n-\epsilon_1\log(qn)/q\bigr) -(2+\epsilon_1)\log(qn)/q\\
       &=n-(2+\epsilon)\log(qn)/q.\qedhere
\end{align*}
\end{proof}

\section{Lower Bounds and an Application of Janson's Inequality}

In this section, we prove New Lower Bounds (2), (3) and~(4).
All three proofs require concepts developed 
in~\cite{ErKiTr91}, starting with a good bound on the expected value of
the balanced independence number.  The following elementary lemma uses
only Markov's inequality.

\begin{lemma}\label{lem:markov-bin}
Suppose $n^{-1}\log^2 n/n\le q\le 1/2$ and $0<\epsilon<1$. Then $a.a.s.$,
\[
\wttbin(P)< t:=  \lceil 2[\log n+\log\log(1/q)]/\log(1/q)\rceil.
\]
\end{lemma}

\begin{proof}
Let $Y$ count the number of balanced independent
pairs of size~$t$.  Then
\[
\bbE[Y]=\binom{n}{t}^2 q^{t^2} < \exp(t(2\log n- t\log(1/q)))=o(1).\qedhere
\]
\end{proof}

With this lemma in mind, for the balance of this section, whenever the value
of $q$ is specified, we set:
\begin{equation}\label{eqn:value-of-t}
t=\lceil (2\log n+\log\log n)/\log(1/q)\rceil.
\end{equation}

Fix a value of $q$, with $t$ then determined 
by~(\ref{eqn:value-of-t}).
A \textit{short pair} is a pair $(\sigma,\sigma')$ where $\sigma$ is a linear order on
a $(t-1)$-element subset of $A$, and $\sigma'$ is a linear order on a $(t-1)$-element
subset of $A'$.  Let $d\in[n]$ and let
$\Sigma=\{(\sigma_j,\sigma'_j):1\le j\le d\}$ be a family of short
pairs.  With the family $\Sigma$ fixed,  we make the following definitions.

For a pair $(a,a')\in A\times A'$ and an integer $j\in[d]$, 
\begin{enumerate}
\item Event $R_j(a,a')$ holds if $a\in\sigma_j$ and 
$b\parallel_P a'$ for all $b\in\sigma_j$ with $h_j(b)<h_j(a)$.
\item Event $R'_j(a,a')$ holds if $a'\in\sigma'_j$ 
and $b'\parallel_P a$ for all $b'\in\sigma'_j$ with $h'_j(b')<h'_j(a')$.
\end{enumerate}
We note that $R_j(a,a')$ holds whenever $a$ is the highest element of $\sigma_j$.
Also, $R'_j(a,a')$ holds whenever $a'$ is the lowest element of $\sigma'_j$.

For a pair $(a,a')\in A\times A'$, let $(a<a')$ be the event that holds when
$a<_P a'$.  Also, let $(a\parallel a')$ be the event that holds when
$a\parallel_P a'$. Now set 
\[
R(a,a')=(a<a')\vee \Bigl(\vee_{j\in[d]}R_j(a,a')\Bigr) \vee\Bigl(\vee_{j\in[d]}R'_j(a,a')\Bigr).
\]
We say that $\Sigma$ \textit{realizes} the pair $(a,a')$ when
$R(a,a')$ holds  

In turn set 
\[
R(P)=\wedge_{(a,a')\in A\times A'} R(a,a').
\]
We say $\Sigma$ is a \textit{short realizer for $P$}
when $R(P)$ holds.  Then we define the \textit{short dimension of $P$}, denoted 
$\sdim(P)$, as the least positive integer $d$ such that there is a family 
$\Sigma=\{(\sigma_j,\sigma'_j): j\in[d]\}$ of short pairs such that
$\Sigma$ is a short realizer of $P$.

We observe that $a.a.s.$, $\sdim(P)\le\dim(P)$.
To see this, let $\cgF=\{L_1,\dots,L_d\}$ be a realizer of $P$.
Then for each $j\in[d]$, let $\sigma_j$ be the linear order consisting of the 
highest $t-1$ elements
of $A$ in $L_j$.  Also, let $\sigma'_j$ be the linear
order consisting of the lowest $t-1$ elements of $A'$ in $L_j$.
Since $a.a.s.$, the balanced independence number of $P$ is less than $t$, it follows that
$a.a.s.$, $\Sigma=\{(\sigma_j,\sigma'_j):j\in[d]\}$ is a short realizer for $P$.
Accordingly, a lower bound on $\sdim(P)$ is also a lower bound on $\dim(P)$.

The next step in the argument for all three lower bounds is to fix a short family $\Sigma=
\{(\sigma_j,\sigma'_j):1\le j\le d\}$ and consider the event
$R(P)$ that holds when $\Sigma$ is a short realizer for $P$.  We will determine
a reasonably accurate upper bound $p_0$ on $\bbP\bigl(R(P)\bigr)$.
The number of short families is less than $n^{2(t-1)d}$ and $d\le n$, so we
can say that $a.a.s.$, $\sdim(P)>d$ if $n^{2tn}p_0=e^{2tn\log n}p_0=o(1)$.

With $\Sigma$ fixed, let $T$ consist of those elements $a\in A$ such that
there is at least one $j\in[d]$ with $a$ the highest element of $A$ in $\sigma_j$.
Then set $M=A-T$.  Analogously, let $T'$ consist of those elements $a'\in A'$
such that there is at least one $j\in[d]$ with $a'$ the lowest element of $\sigma'_j$.
Then set $M'=A'-T'$.  Also, set $s=t-2$.

We note that $R(a,a')$ holds whenever $a\in T$ or $a'\in T'$.  Accordingly,
\[
R(P)=\wedge_{(x,x')\in M\times M'} R(x,x').
\]

Our next goal will be to determine a bound on $d$ that forces $\bbP(R(P))$ to
be exponentially small.  Some additional notation and terminology is required.
We describe this notation in full detail for $M$.  The notation for $M'$ is dual.

When $j\in[d]$, $x\in M$, and $x\in\sigma_j$, we let $h_j(x)$ count the number
of elements $y\in\sigma_j$ with $y$ higher than $x$ in $\sigma_j$.
By convention, we set $h_j(x)=\infty$ if $x\not\in\sigma_j$.
It is natural to view the quantity $h_j(x)$ as the \textit{height of
$x$} in $\sigma_j$. For an integer $i\in[s]$, we then let $\mu_i(x)$ count
the number of $j\in[d]$ with $h_j(x/W)=i$.  We view $\mu_i(x)$ as
the \textit{multiplicity of $x$ for height~$i$}.
Then define the quantity $w(x)$ by setting
\[
w(x)=\sum_{i=1}^s \mu_i(x) 2^{1-i}.
\]
We view the quantity $w(x)$ as the \textit{weight of $x$}.
Note that
\[
\sum_{x\in M} w(x) < 2d.
\]
Since $d\le n$, there is a subset $M_0\subset M$ with
$|M_0|=m/2$ such that $w(x)<4d/m$ for every $x\in M_0$.

The preceding discussion is followed in a dual manner to determine
a subset $M'_0\subset M'$ with $|M'_0|=m/2$ so that
$w(x')< 4d/m$ for every $x'\in M'_0$.  Set $\cgI=M_0\times M'_0$.
In the analysis to follow, we will need the following elementary
fact.  It is stated formally, as we will need it again in the
following section.

\begin{proposition}[Weight-Shift]\label{pro:weight-shift}
If $0\le q\le 1/2$ and $i\in N$, then $1-q^i<(1-q^{i+1})^2$.
\end{proposition}

For an event $E$ in a probability space $\Omega$, we use the notation
$\overline{E}$ to denote the event that holds when $E$ fails. 
Let $(x,x')\in\cgI$.  We consider the events in the family
$\{\overline{R}_j(x,x'):j\in[d]\}$.  As explained in~\cite{ErKiTr91}, these
events are positively correlated, i.e., when $j$ and $k$ are
distinct integers in $[d]$,
\[
\bbP(\overline{R}_j(x,x')|\overline{R}_k(x,x'))\ge\bbP(\overline{R}_j(x,x'))).
\]

With the convention that $1-q^{\infty}=1$, it follows that:
\begin{align*}
\bbP(\wedge_{j\in[d]}\overline{R}_j(x,x')
     &\ge\prod_{j\in[d]} (1-q^{h_j(x)})&& \text{Using correlation.}\\
     &\ge\prod_{i\in[s]} (1-q^i)^{\mu_i(x)}&& \text{Definition of multiplicity.}\\
     &\ge(1-q)^{w(x)}&&\text{Using Proposition~\ref{pro:weight-shift}.}\\
     &\ge (1-q)^{4n/m}.
\end{align*}
Analogously, we have:
\[
\bbP\bigl(\wedge_{j\in[d]}\overline{R}'_j(x,x')\ge(1-q)^{4n/m}.
\]
It follows that
\begin{equation}\label{eqn:key-bound}
\bbP\bigl(R(x,x'))\le 1-q(1-q)^{8n/m}.
\end{equation}

\subsection{Applying the Janson Inequality}
 
We will use the Janson inequality in the proofs of New Lower Bounds~(2) and~(3).
Here is the set up for this result, following (essentially) the presentation 
in Chapter~8 of Alon and Spencer~\cite{AloSpe16}.  The text~\cite{JaLuRu00} by 
Janson, \L uczak and Rucinski is cited for the proof. 

Let $\cgI$ be a finite set and let $\{F_i:i\in\cgI\}$ be a finite
family of events in a probability space $\Omega$. When $i$ and $j$ are distinct
elements of $\cgI$, we write $F_i\sim F_j$ when $F_i$ and $F_j$ are dependent.
Also, we set 
\[
\Delta=\sum\{\bbP(F_i\wedge F_j):(i,j)\in\cgI\times\cgI, F_i\sim F_j\},
\]
and
\[
\mu=\sum_{i\in\cgI}\bbP(F_i).
\]

Here is the statement of the Janson inequality we will apply.

\begin{theorem}[Janson Inequality]\label{thm:Janson}
Let $\{F_i:i\in\cgI\}$ be a finite family of events with
$\bbP(F_i)\le 1/2$ for all $i\in \cgI$.
If $\Delta\le\mu$, then 
\[
\bbP\Bigl(\wedge_{i\in \cgI}\overline{F_i}\Bigr)\le  \exp(-\mu/2).
\]
\end{theorem}

\begin{lemma}[New Lower Bounds~(2) and (3)] 
If $n^{-4/5} \le q\le \frac{1}{8}n^{-1/4}\log^3 n$, then $a.a.s.$,
\begin{equation*}
\dim(P)>
\begin{cases}
 n- 32 (n\log n/q)^{1/2}  &\text{if } n^{-4/5}\le q\le (16)^{1/3}n^{-1/3}\log^{1/3} n.\\
 n- 8qn     &\text{if }(16)^{1/3}n^{-1/3}\log^{1/3}n\le q \le \frac{1}{8}n^{-1/4}\log^3 n.
\end{cases}
\end{equation*}
\end{lemma}

\begin{proof}
For this range, we note that $t\le 9$.
For every $(x,x')\in\cgI$, we have
an event $F(x,x')$ that holds when $\Sigma$ \emph{fails} to realize the pair
$(x,x')$.  Note that for $F(x,x')$ to hold, we need $x\parallel_P x'$, so
$\bbP(F(x,x'))\le q\le 1/2$.  We also observe that $\bbP(F(x,x'))\ge q(1-q)^{8n/m}$. 
Set $m=qnz$ where $z\ge 8$.  In general, we will have $z\rightarrow\infty$, but the restriction
$z\ge8$ is enough to imply that $(1-q)^{8n/m}\ge 1/4$.
It follows that $\mu$, the exected number of pairs that fail, is
at least $m^2q/16$.

When $x,y\in M_0$ and $x',y'\in M'_0$, we observe that
$F(x,x')\sim F(y,y')$ if and only if $|\{x,x'\}\cap\{y,y'\}|=1$.
Furthermore, when $F(x,x')\sim F(y,y')$, and event $F(x,x')\wedge F(y,y')$ holds,
we must have $x\parallel_P x'$ and $y\parallel_P y'$.  These two events are
independent and each has probability $q$.  It follows that
\[
\bbP\bigl(F(x,x')\wedge F(y,y')\bigr)\le q^2.
\]
There are are $2(m/2)(m/2)$ such pairs so $\Delta\le m^2q^2/2$.  To
apply the Janson inequality, we need $\Delta\le \mu$, but this simply requires
$q\le 1/8$.  We conclude that
\[
\bbP\Bigl(\bbR(\Sigma)\Bigr)\le\bbP\Bigl(\wedge_{(x,x')\in\cgI}\overline{F(x,x')}\Bigr)
                           \le \exp(-\mu/2)\le \exp(-m^2q/32).
\]

Recall that the number of short families is less than $n^{2tn}\le \exp(18n\log n)$,
since $t\le 9$.  Noting that $18\cdot 32= 576<2^{10}$,
can conclude that $a.a.s.$, $\dim(P) > d$ if
$2^{10}n\log n\le m^2 q$.  Since $m=qnz$, this
becomes $2^{10}\log n\le q^3 n z^2$.  This requires
\begin{equation}\label{eqn:z-equation}
z \ge 32\bigl(\log n/(q^3 n)\bigr)^{1/2}. 
\end{equation}
When $z=8$, this inequality holds when $q\ge (16)^{1/3}n^{-1/3}\log^{1/3}n$, and this
completes the proof of New Lower Bound~(3).

Now we assume that $q\le (16)^{1/3}n^{-1/3}\log^{1/3} n$.  Now we treat 
inequality~(\ref{eqn:z-equation}) as an equation, i.e.,
we set $z=32\bigl(\log n/(q^3n)\bigr)^{1/2}$.  The equation $m=qnz$ is equivalent to
$m=32(n\log n/q)^{1/2}$, and with this observation,
the proof of New Upper Bound~(2) is complete.
\end{proof}

\subsection{A Family of Independent Events}

To obtain a proof of New Lower Bound~(4), we simply update the original argument 
in~\cite{ErKiTr91} as given on pages 262--268.  We have elected not to repeat the 
details of this argument.  Instead, we will provide only an outline of the steps 
to be taken, with notational changes made to agree with our treatment here

\begin{enumerate}
\item We identify a subset $\cgJ$ of $M\times M'$ with $|\cgJ|=m^3/(72ns^2)$.
\item For each $(x,x')\in\cgI$, we determine an event $E(x,x')$ such that
$R(x,x')\subseteq E(x,x')$ and $\Pr(E(x,x'))\ge q(1-q)^{24n/m}$.
\item Events in the family $\{E(x,x'):(x,x')\in\cgJ\}$ are independent.
\end{enumerate}

Since the events in $\cgJ$ are independent, it follows that
\begin{equation}\label{eqn:p_0-ub}
\bbP(R(P)) \le\Bigl[1-q(1-q)^{24n/m}\Bigr]^{\frac{m^3}{72ns^2}}.
\end{equation}

Inequality~(\ref{eqn:p_0-ub}) provides an upper bound on $p_0$, the maximum
value of $\bbP(R(P))$.  To show that $a.a.s.$, $\dim(P)>d=n-m$, it
suffices to require that:
\[
e^{2tn\log n}  \bigl[1-q(1-q)^{24n/m}\bigr]^{m^3/72ns^2}=o(1)
\]
In the range we consider, it will always be the case that $q(1-q)^{24n/m}=o(1)$.
With this restriction, the preceding inequality holds if:

\begin{equation}\label{eqn:master}
n^2t^3\log n=o\Bigl(q(1-q)^{24n/m}m^3\Bigr).
\end{equation}

In the proof of the next lemma, we will refer to~(\ref{eqn:master})
as the ``master inequality.''

\begin{lemma}[New Lower Bound (4)]
If $\frac{1}{8}n^{-1/4}\log^3 n\le q\le 1/2$ and $z=\log n + 4\log q - 8\log\log n$,  then
$a.a.s.$, $\dim(P)\ge n-24qn/z$.
\end{lemma}

\begin{proof}
Set $m=24qn/z$.  The lower bound on $q$ implies
$qn^4\ge \log^{11} n$.  It follows that $z\ge 3\log\log n$ so that $m=o(qn)$.
With this value of $m$, we can safely approximate $(1-q)^{24n/m}$ as $e^{-z}$.
Accordingly, the master inequality
becomes $n^2t^3\log n=o(qe^{-z}q^3n^3/z^3)$,  which is equivalent to $t^3z^3e^z\log n=o(q^4n)$.
We note that $e^z= qn^4/\log^8 n$.  So the master inequality holds if
$t^3z^3\log n=o(\log^8 n)$.  However, $t\le  3\log n$.  Furthermore,
$z <\log n$.  It follows that $t^3z^3\log n=O(\log^7 n)$, so that
the master inequality holds.
\end{proof}

We observe that there is a threshold occurring when $q\sim n^{-1/4}$. When $q$ is
below this threshold, the Janson inequality approach gives a better result, and
when $q$ is above this threshold, the original approach using a family of
independent events is better.

\section{Generalized Latin Rectangles and the Euler Product Function}

In this section, we prove New Upper Bounds~(2) and~(3).
The arguments require a ``one-sided'' reformulation of
dimension, using the same approach (at least one half of it) taken
in the last section.  Let $P\in\bbB$ and let 
$\cgF=\{\sigma_1,\dots,\sigma_d\}$ be a non-empty family of linear orders 
such that for each $j\in[d]$, $\sigma_j$ is a linear order on a non-empty 
subset of $A$.  Now there is no restriction on the size of these linear
orders.  For a pair $(a,a')\in A\times A'$, we say that $\cgF$ \textit{realizes}
$(a,a')$ if either (1)~$a<_P a'$ or (2)~there
is some $j\in[d]$ with $a\in\sigma_j$ and $b\parallel_P a'$ for
all $b'\in\sigma_j$ with $h_j(b)<h_j(a)$.  In turn, we say $\cgF$ is a 
\textit{one-sided realizer} for $P$ when $\cgF$ realizes 
$(a,a')$ for all pairs $(a,a')\in A\times A'$.

Clearly, $\dim(P)$ is the least $d\ge1$ for which $P$ has a one-sided realizer
of size~$d$. Our strategy for proving New Upper Bounds~(1) and~(2)
will be to design a single candidate family $\cgF=\{\sigma_1,\dots,\sigma_d\}$ and
show that $a.a.s.$, this family is a one-sided realizer of a poset $P$.
To implement this strategy, we must pause to establish a connection with a classic 
concept in combinatorics.
  
\subsection{Generalized Latin Rectangles}

Recall that when $m$ and $s$ are integers with $1\le s\le m$,
an $s\times m$ array (matrix) $R$ is called a \textit{latin rectangle}
when (1)~each row of $R$ is a permutation of the integers in
$[m]$, and (2)~the entries in each column of $R$ are
distinct.  As is well known, if $2\le s\le m$, an $(s-1)\times m$ 
latin rectangle $R$ can always be extended to an $s\times m$ latin 
rectangle by adding a new row.

Now let $(m,r,s)$ be a triple of positive integers.  An $s\times (rm)$
array $R$ of integers from $[m]$ will be called an  $(m,r,s)-\GLR$ (where
$\GLR$ is an abbreviation for \textit{generalized latin rectangle})
when the following conditions are met:

\begin{enumerate}
\item In each row of $R$, each integer in $[m]$ occurs exactly
$r$ times.
\item In each column $C$ of $R$, the $s$ integers occuring
in column $C$ are distinct.
\item For each distinct pair $i,j\in\{1,2,\dots,m\}$, there is at
most one column $C$ in $R$ for which $i$ is below $j$ in column~$C$.
\end{enumerate}

Note that when $r=1$, the third requirement is not part of the
traditional definition for a latin rectangle.  However, it will be soon
be clear why we want this additional restriction in place.

Here is an example of a $(9,2,3)$-GLR.

\begin{equation*}
\setcounter{MaxMatrixCols}{20}
\begin{bmatrix}
1 & 1 & 2 & 2 & 3 & 3 & 4 & 4 & 5 & 5 & 6 & 6 & 7 & 7 & 8 & 8 & 9 & 9\\
8 & 9 & 9 & 1 & 1 & 2 & 2 & 3 & 3 & 4 & 4 & 5 & 5 & 6 & 6 & 7 & 7 & 8\\
3 & 6 & 4 & 7 & 5 & 8 & 6 & 9 & 7 & 1 & 8 & 2 & 9 & 3 & 1 & 4 & 2 & 5
\end{bmatrix}
\end{equation*}
The reader may note that it is impossible to extend this array to
a $(9,2,4)-\GLR$.  More generally, we have the following natural extremal
problem: For a pair $(m,r)$ of positive integers,  find the the largest
integer~$s=f(m,r)$ for which there is an $(m,r,s)-\GLR$.  Trivially,
$f(m,r)\ge1$.   

\begin{lemma}\label{lem:GLR-1}
Let $m,r,s$ be positive integers with $s\ge2$.  If $f(m,r)\ge s$, 
then $rs(s-1)\le 2(m-1)$.
\end{lemma}

\begin{proof}
Suppose that $R$ is an $(m,r,s)-\GLR$.
There are $rm$ columns in $R$ and for each column $C$ in $R$,
there are $s(s-1)/2$ ordered
pairs $(i,j)$ where $i$ is below $j$ in column~$C$. The
last two conditions in the definition of an $(m,r,s)$-GLR force
$rms(s-1)/2\le m(m-1)$, so that $rs(s-1)\le 2(m-1)$.
\end{proof}

Lower bounds on this extremal problem are more challenging, but
we will give an explicit construction which is sufficient for
our purposes.  If $R$ is an $(m,r,s)-\GLR$, we say $R$ is \textit{resolvable}
if it consists of $r$ latin rectangles placed side by side, i.e.,
each row partitions into $r$ blocks of consecutive elements
and each block is a permutation of $[m]$.
The following lemma is an elementary extension of the classical
result for latin rectangles, and we only outline the proof, leaving
the details as an exercise for students.

\begin{lemma}\label{lem:GLR-2}
Let $m,r$ and $s$ be positive integers with $s\ge2$.  If 
$m> 2rs^3$, then there is an $(m,r,s)-\GLR$.
\end{lemma}

\begin{proof}
Consider the problem of adding a last row to a resolvable
$(m,r,s-1)-\GLR$.  Proceeding block by block, we 
have a balanced $(m,m)$ bipartite graph $G$ with positions
$1$ through $m$ on one side of $G$ and sets of allowable
choices for each of the $m$ positions on the other side.  Clearly, 
the most challenging case in completing the last row is the last
block.

Consider one of the $m$ columns in the last block, and let
$x$ be one of the $s-1$ integers that already occurs in this column.
Then $x$ is over $(r-1)s(s-1)/2$ other integers in the first $r-1$
blocks, and $x$ is over $(s-1)(s-2)/2$ other integers in the last
block.  When $x$ is the lowest element in the column, then $x$
itself is not allowable.  It follows that the number of allowable
choices is:

\begin{equation}\label{eqn:cell-size}
m -(r-1)s(s-1)^2/2 -(s-1)^2(s-2)/2 -1
\end{equation}
Note that the inequality $2rs^3<m$ implies that the quantity in
inequality~(\ref{eqn:cell-size})
is at least $m/2$.   A parallel argument shows that
each of the integers in $[m]$ belongs to at least $m/2$ of the
sets of allowable choices.  It is an
immediate consequence of Hall's theorem that
a balanced bipartite graph $G$ with $2m$ vertices
and minimum degree $\delta(G)\ge m/2$ has a complete matching.
Futhermore, a matching in $G$ provides a legal way to complete the
last row.
\end{proof}

Let $m,r,s$ be integers, and let $R$ be an $(m,r,s)-\GLR$.  We set $d=mr$ and
$n=d+m$.  We fix an arbitrary $d$-element subset $T=\{a_1,\dots,a_d\}$ of $A$
and set $M=A-T=\{x_1,\dots,x_m\}$.   
We construct a family $\cgF=\{\sigma_1,\dots,\sigma_d\}$ of linear orders on
$(s+1)$-element subsets of $A$ as follows.
For each $j\in[d]$, we set $h_j(a_j)=0$. Then for each pair $(i,j)\in[s]\times[d]$,
we set $h_j(x_\alpha)=i$ when the integer in row $i$ and column $j$ of $R$
is $\alpha$.  

When $(a,a')\in A\times A'$, it is obvious that $\cgF$ realizes $(a,a')$ if
$a\in T$.  Let $X$ be the random variable counting the number of
pairs $(a,a')\in A\times A'$ for which $\cgF$ fails to realize $(a,a')$.
Since $a$ must belong to $M$, the expected value of $X$ is given by:
\begin{equation}\label{eqn:ev-partial-phi}
E[X]=nmq\bigl[\prod_{i=1}^s(1-q^i)\bigr]^r.
\end{equation}

Note that the expression $\prod_{i=1}^s(1-q^i)$ is a partial product
of the Euler product function $\phi(q)=\prod_{i=1}^{\infty}(1-q^i)$. 
Using the weight-shift propososition~\ref{pro:weight-shift} from the
preceding section, it follows that
\[
(1-q^s)\prod_{i=1}^s(1-q^i) <\phi(q) < \prod_{i=1}^s(1-q^i).
\]

\begin{lemma}[New Upper Bound (3)]
If $n^{-1/3}\le q\le 1/2$ and $z=n^2q\log(1/\phi(q))$, then $a.a.s.$,
$\dim(P)<n-n\log(1/\phi(q))/\log z$.
\end{lemma}

\begin{proof}
We will only consider values of $m=n-d$ with $m\ge qn/\log^{10}n$.  With this
restriction $r=d/m= n/m-1\le \log^{11} n/q$.  The requirement for the
existence of a $(m,r,s)-\GLR$ is $m>2rs^3$.  With the restrictions on $q$ and
$m$, the requirement is met when $s=n^{1/10}$.
For such a large value of $s$, we are safe if we estimate
$\prod_{i=1}^s(1-q^i)$ by $\phi(q)$.  Accordingly, we can conclude that
$a.a.s.$, $\dim(P)\le d$ if
\begin{equation}\label{eqn:1}
E[X]=nmq\bigl(\phi(q)\bigr)^r\rightarrow 0.
\end{equation}

We note that $\phi(q)<1$, $r=d/n=n/m-1$ and $p=1-q\le 1/2$.  It follows
that inequality~(\ref{eqn:1}) is equivalent to:
\begin{equation}\label{eqn:2}
nmq=o\bigl(e^{n\log(1/\phi(q))/m}\bigr).
\end{equation}

With $z$ set at $n^2q\log(1/\phi(q)$, we note that $z$ and $\log z$ tend to
infinity with $n$. Now set $m=n\log(1/\phi(q))/\log z$.  Since
$e^{\log z} =z = n^2q\log(1/\phi(q))$, inequality~(\ref{eqn:2})
is equivalent to:
\begin{equation}\label{eqn:3}
n^2q\log(1/\phi(q))=o\bigl(n^2q\log(1/\phi(q))\log z\bigr).
\end{equation}

Clearly, this last inequality is satisfied.
\end{proof}

We are reasonably confident that New Upper Bound~(3) is asymptotically
correct. 

\begin{lemma}[New Upper Bound (2)]
If $n^{-1/2}\log n <q\le n^{-1/3}$, then $a.a.s.$,
$\dim(P) <n -qn/(2\log(qn))$.
\end{lemma}

\begin{proof}
We note that for all pairs $(m,r)$, there is a $(m,r,1)-\GLR$. Of course, this simply
means that we put every element of $M$ in second position $d/m=r$ times.
Setting $s=1$ in equation~(\ref{eqn:ev-partial-phi}), we can conclude that
$a.a.s.$, $\dim(P) \le d$ if
\begin{equation}\label{eqn:4}
nmq(1-q)^{d/m}\rightarrow 0.
\end{equation}
Again, we note that $(1-q)^{d/m}= (1-q)^{n/m}/(1-q)\ge (1-q)^{n-m}/2$.
Now set $m = qn/(2\log(qn))$.  Noting that $2\log(qn)\rightarrow\infty$, 
we have 
\[
(1-q)^{n/m}=e^{-qn/m}=e^{2\log(qn)}=n^2q^2.
\]
Therefore, inequality~(\ref{eqn:4}) is equivalent to:
\begin{equation}\label{eqn:5}
n^2q^2=o\bigl(\log(qn) n^2q^2\bigr).
\end{equation}

This last equation holds since $qn\ge\log^2n$ so that $\log(qn)\rightarrow\infty$.
\end{proof}

\section{Applications to the Extremal Problems}\label{sec:eps}

We now return to Question~1 and the problem of finding the correct exponent on
the function $\sa(c)$.  Previously, we reported that we had been
able to use the asymmetric form of the Lov\'{a}sz local lemma to 
raise the lower bound on the exponent of $c$ in $\sa(c)$ from $4/3$ to 
$3/2$.  This unpublished result
was presented at several conferences and seminars and was proved by using the
local lemma to find a rare poset with independence number less than~$2$ and only
moderately large standard example number.  Lemma~\ref{lem:matchings-2} was then 
used to determine the dimension of such a poset.

However, our new bounds allow us to obtain a simple proof of this
same improvement.  Consider the value $q=n^{-1/3}$.  If a bipartite poset
$P\in\bbB(n)$ contains the standard example $S_d$, then its clique size
is at least $d/2$.   Since the clique size of $P\in\Omega(n,p)$ is
$a.a.s.$, less than $3q^{-1}\log(qn)=2n^{1/3}\log n$, it follows that $a.a.s.$, 
$\se(P) <4n^{1/3}\log n$.  On the other hand, with $q=n^{-1/3}$, we know that
$a.a.s.$,
\[
\dim(P)\ge n- 32\bigl(n\log n/q)^{1/2}= n-32n^{2/3}\log^{1/2} n.
\] 
Setting $c=32n^{2/3}\log^{1/2} n$, the upper bound on $\se(P)$ forces $f(c)=
\Omega\bigl(c^{3/2}/ \log ^{3/4} c\bigr)$.

\subsection{Progress on Question 2 for Posets}

We begin with the following elementary result, for which we only outline the
proof. Ironically, when applied it will be for $p<1/2$.

\begin{lemma}\label{lem:matchings-AA'}
Suppose $n^{-1}\log^2 n \le q\le 1/2$.  Then $a.a.s.$, $A$ and
$A'$ can be matched.
\end{lemma}

\begin{proof}
Clearly, it is enough to prove the lemma when $q=n^{-1}\log^2 n$.
The basic idea is to show that $a.a.s.$, Hall's matching condition is satisfied.
For a subset $W\subset A$, let $N(W)$ consist of all elements of $A'$ that
are incomparable with at least one element of $W$.  We want to show
that $a.a.s.$, $|N(W)|\ge|W|$ for every subset $W\subseteq A$.  First,
we take care of the case when $W$ is very small or very large.

Set $r=\log n/5$.
Consider the events $E_1$ that holds if there is some $a\in A$
incomparable with fewer than $r$ elements of $A'$.  Dually, event $E_2$ holds
if there is some element $a'\in A'$ incomparable with fewer than $r$ elements
of $A$.  Simple counting shows that $\bbP(E_i)=o(1)$ for $i=1,2$.

Now consider the event $F$ that holds if
there is some set $W$ with $r\le|W|\le n-r$ such that $|(N(W)|<|W|$.  Then
we show that $\bbP(F)=o(1)$.  Readers will note that this part of the
proof is very similar to the proof (which is provided) of
Lemma~\ref{lem:matchings-defect}.

Once we have shown that $\bbP(E_1)= \bbP(E_2)= o(1)$ and $\bbP(F)=o(1)$, the
proof is complete.
\end{proof}

Recall that $f(d,n)$ is the maximum
value of $\dim(P)$ among all posets on $n$ points with standard example size less
than $d$.

\begin{theorem}
For all $d\ge 3$,
\[
f(d,n)\geq \frac{n^{1- \frac{2d-1}{d(d-1)}}}{8\log n}.
\]
\end{theorem}

\begin{proof}
Fix a value of $d\ge3$.
We work in the space $\Omega(n,p)$ with
\begin{equation}\label{eqn:p}
   p  = n^{-\frac{2d-1}{d(d-1)}}.
\end{equation}
We note that $p<1/2$.  Now set
\[
\epsilon=\frac{2d(d-1)}{2d-1}-2\quad\text{and}\quad t=2\log n/p.
\]
We note that $\epsilon>0$.

Using Lemma~\ref{lem:matchings-AA'}, we know that $a.a.s.$, the sets
$A$ and $A'$ can be matched.  Furthermore, the following claim is just
Lemma~\ref{lem:markov-bcn} stated in complementary form.

\medskip
\noindent
\textbf{Claim 1.}\quad For every $\epsilon$ with $0<\epsilon < 1$, $a.a.s.$,
$\wttbin(P)<(2+\epsilon)\log(pn)/p$.

\medskip
Therefore, $a.a.s.$,
\[
\wttbin(P)<(2+\epsilon)\log(pn)/p=(2+\epsilon)\frac{(2d-1)\log n}{(d-1)p}=2\log n/p=t.
\]

Next, we need the following technical claim. 

\medskip
\noindent
\textbf{Claim 2.}\quad
If $P$ is a poset in $\bbB(n)$ and $\wttbin(P)<t$, then 
$P$ does not contain a bipartite
subposet $Q=V\cup V'$ with $|V|=|V'|=2t$ satisfying the following 
condition: The elements of $V$ and $V'$ can be labeled
as $V=\{v_1,\dots,v_{2t}\}$ and 
$V'=\{v'_1,\dots,v'_{2t}\}$ such that there is a linear 
extension $L$ of $Q$ with $v_i>v'_i$ in $L$ for each $i\in[2t]$.

\begin{proof}
We argue by contradiction.  After a relabeling, we 
may assume that $v_1>\dots>v_{2t}$ in $L$. However, this 
implies that $v_i>v'_j$ in $L$ whenever $1\le i\le t$ and 
$t+1\le j\le 2t$, which implies that $\wttbin(P)\ge t$.
\end{proof}

Next, we let $X$ be the random variable that counts the 
number of copies of the standard example $S_d$ in $P$. Then the expected 
value of $X$ is given by
\[
\bbE[X]=\binom{n}{d}^2 d! (1-p)^d p^{d(d-1)} < \frac{n^{2d}}{6}p^{d^2-d}= n/6.
\]

Let $E$ be the event that occurs when $X > n/4$. Then $\bbP(E)<2/3$. It
follows that there is a poset $P\in\Omega(n,p)$ such that (1)~$A$ and $A'$
can be matched; (2)~$\wttbin(P) < t$; and 
(3)~the number of copies of the standard example $S_d$ in $P$ is at most $n/4$.
Let $A=\{a_1,a_2,\dots,a_n\}$ and 
$A'=\{a'_1,a'_2,\dots,a'_n\}$ be labelings that evidence a 
matching between $A$ and $A'$. Without loss of generality, we may assume that any copy of 
$S_d$ contained in $P$ (there are at most $n/4$ of them) contains some point 
in $\{a_i:n/2<i\le n\}\cup\{a'_i:n/2<i\le n\}$. Hence there are 
\textit{no} copies of $S_d$ in the bipartite subposet 
$Q=B\cup B'$ where $B=\{a_i:1\leq i\leq n/2\}$ and 
$B'=\{a'_i:1\leq i\leq n/2\}$. Note that we have $|Q| = n$.

Now let $\cgF$ be any family of linear extensions of $Q$ which is 
a realizer of $P$. Then $\cgF$ must reverse the pairs in 
$\{(a_i,a'_i):1\leq i\leq n/2\}$. However, in view of the claim, no
linear extension can reverse $2t$ of these pairs. 
we conclude that
\[
\dim(P)\ge \dim(Q)\geq \frac{n}{4t}= \frac{pn}{8\log n}=
\frac{n^{1- \frac{2d-1}{d(d-1)}}}{8\log n}.\qedhere
\]
\end{proof}

\section{Some Comments on Open Problems}\label{sec:close}

We view the problem of determining the expected value of $\dim(P)$ for posets 
in $\Omega(n,p)$ when $n^{-4/5}< q\leq n^{-1/3}$ to be a real challenge.
We suspect that our New Upper Bounds are near the truth, but we
cannot rule out the possibility that for almost all $P$, if
the short dimension of $P$ is $d$ as evidenced by a short
realizer $\Sigma$ of size $d$, then $\bbP\bigl(\Sigma(R)\bigr)$
is very small.

Second, although we believe we \emph{know} the expected value
of $n-\dim(P)$ to within a $1+o(1)$ multiplicative factor when 
$1/2\le p< 1-n^{-1/3}$ as specified by New Upper Bound~(1), the challenge
is that there are other constructions besides generalized latin rectangles
that achieve the same bound.

For the first extremal problem, stability analysis, we continue to think 
it likely that the correct exponent for $c$ in the function $\sa(c)$ is~$2$.  As we have
noted, this would be verified if the upper bound on $\dim(P)$ when
$q=n^{-1/2}$ is correct to within a poly-log multiplicative factor
on $n-\dim(P)$.  Alternatively,  one could revisit 
the proof given in~\cite{BHPT16} and try to lower the exponent in the 
inequality $\sa(c)=O(c^2)$.  Success in this effort would of course
imply that our upper bounds on $\dim(P)$ are not as good as we think.

For the second extremal problem, it would be very interesting to 
show that for each $d\ge3$, there is a constant $c_d$, with $0<c_d<1$, such that
$f(n,d)<n^{c_d}$, although it is not clear that such a constant
exists, even when $d=3$.

\section{Acknowledgment}

The authors would like to thank Noga Alon, Alan Frieze, and Tomasz \L uczak
for very helpful communications concerning second moment methods, the concept of 
defect, and Talagrand's inequality.  As noted previously, the heart of the
proof of Lemma~\ref{lem:Luczak} was provided by \L uczak in a personal communication.

\end{document}